\documentclass[12pt]{article}
\usepackage{verbatim}
\usepackage{latexsym}
\usepackage{amsmath,amsthm,amsfonts}
\usepackage{eucal}
\usepackage{cases}

\usepackage{mathtext}
\usepackage[cp1251]{inputenc}
\usepackage[T2A]{fontenc}
\usepackage[dvips]{graphicx}
\usepackage{amsmath}
\usepackage{amssymb}
\usepackage{amsxtra}
\usepackage{latexsym}
\usepackage{ifthen}
\usepackage{color}

\newtheorem{theorem}{Theorem}
\newtheorem{definition}{Definition}[section]
\newtheorem{lemma}{Lemma}[section]

\numberwithin{equation}{section}

\begin{document}

\title{Finite speed of propagation for the thin film equation in spherical geometry}

\author{Roman M. Taranets \\
 Institute of Applied Mathematics and Mechanics of the NASU,\\
 Sloviansk, Ukraine, taranets\_r@yahoo.com}

\maketitle
\begin{abstract}
We show that a double degenerate thin film equation, which originated from modeling of viscous
coating flow on a spherical surface,  has finite speed of propagation for nonnegative strong
solutions and hence there exists an interface or free boundary separating the regions where solution $u>0$
and $u=0$. Using local entropy estimates we also obtain an upper bound for the rate of the interface propagation.
\end{abstract}

\textbf{UDC} 517.953


\section{Introduction}

In this paper, we study a particular case of the following doubly degenerate fourth-order parabolic equation
\begin{equation}\label{eq:physical eqn}
u_{t}+ [u^{n} (1-x^{2} ) (a - b x + c (2u + ((1-x^{2})u_{x} )_{x} )_{x} ) ]_{x}=0 \text{ in } Q_T,
\end{equation}
where $u(x,t)$ represents the thickness of the thin film, the dimensionless parameters $a$, $b$ and $c$ describe
the effects of gravity, rotation and surface tension, $Q_T = \Omega \times(0,T )$,
$n > 0$, $T > 0$, and $\Omega=(-1,1)$. For $n=3$ (no-slip regime) this equation describes
the dynamics of a thin viscous liquid film on the
outer surface of a solid sphere. For $n = 2$  the classical Navier slip condition
is recovered. On the other hand, parameter ranges  $n \in (0,2)$ ($n \in (2,3)$) in the equation (\ref{eq:physical eqn}) correspond
to strong (weak) wetting slip regimes. More general dynamics of the liquid film for the case
when the draining of the film due to gravity was balanced by centrifugal
forces arising from the rotation of the sphere about a vertical axis and by capillary forces due
to surface tension was considered in \cite{kangcoatingsphere}. In addition, Marangoni
effects due to temperature gradients were taken into account in \cite{KNC-17}.
The spherical model without the surface tension and Marangoni effects was studied
in \cite{TH-10,Wil-94}.

We are interested in time evolution of
the support of non-negative strong solutions to
\begin{equation}\label{eq:main eqn}
u_{t}+\left( (1-x^{2} ) |u|^{n}( (1-x^{2} )u_{x} )_{xx}\right)_{x}=0.
\end{equation}
Equation (\ref{eq:main eqn}) is a particular case of (\ref{eq:physical eqn})
with $a = b =0$ with an absence of the second-order diffusion term. Existence of weak solutions for (\ref{eq:main eqn}) in a weighted
Sobolev space was shown in \cite{KSZ-16} and  existence of more regular non-negative strong solutions
of (\ref{eq:main eqn}) was recently proved in \cite{Tar-17}.
Unlike the classical thin film equation
\begin{equation}\label{eq:Bernis eqn}
u_{t}+\left(|u|^{n}u_{xxx}\right)_{x}=0,
\end{equation}
the qualitative behavior of solutions for double degenerate thin-film equation (\ref{eq:main eqn}) is still not well understood.
Note that the model equation (\ref{eq:Bernis eqn}) describes the coating flow of
a thin viscous film on a flat surface under the surface tension effect.
Depending on the value of the parameter $n$ non-negative solutions of this equation posses some
properties. For example, in 1990, Bernis and Friedman \cite{bernis1990higher}
defined and constructed non-negative weak solutions of the equation (\ref{eq:Bernis eqn}) when
$n \geqslant 1$, and it was also shown that for $n\geqslant 4$, with a uniformly positive
initial condition, there exists a unique positive classical solution.
Later on, in 1994, Bertozzi et al. \cite{bertozzi1994singularities} generalised this positivity
property  for the case $n \geqslant \frac{7}{2}$. In 1995, Beretta et al. \cite{beretta1995nonnegative}
proved the existence of non-negative weak solutions for the equation (\ref{eq:Bernis eqn}) if
$n > 0$, and the existence of strong ones for $0< n < 3$. Also, they could show that
this positivity-preserving property holds at almost every
time $t$ in the case $n \geqslant 2$.  This positivity-preservation result was generalized for a cylindrical
surface was obtained in \cite{Ch10}. Furthermore, for $n \geqslant \frac{3}{2}$
the solution's support to (\ref{eq:Bernis eqn}) is non-decreasing in time, and the support remains
constant if $n \geqslant 4$. The existence (nonexistence) of compactly supported
spreading source type solution to (\ref{eq:Bernis eqn}) was demonstrated
for $0 < n < 3$ $(n \geqslant 3)$ in \cite{BPW-92}. One of interesting
qualitative properties of non-linear parabolic thin film equations is finite speed
of support propagation that is not the case when the parabolic equation is a linear one.
This property was first shown in \cite{B6} if $0 < n < 2$, and in  \cite{B7,HulShish}
if $2 \leqslant n < 3$ for non-negative strong solutions of  (\ref{eq:Bernis eqn}).
A similar result on a cylindrical surface was obtained in \cite{CT-12}.


Our main result for the thin film equation on the spherical surface is the finite speed of the
interface propagation in the special case of the strong slip regime $n \in (1,2)$.
Proof of the finite speed of propagation property is based on local entropy estimate
and Stampacchia's lemma. Moreover, we obtain an upper bound the time evolution of the support as:
$\Gamma(t) \leqslant C_0 t^{\frac{1}{n +4}}$. This bound coincides with  the asymptotic
behaviour of self-similar type solutions to (\ref{eq:Bernis eqn})
(see \cite{BPW-92}).

\section{Main result}

We study the following thin film equation
\begin{equation}\label{A-1}
u_{t}+\left((1-x^{2}) |u|^{n}\left((1-x^{2})u_{x}\right)_{xx}\right)_{x}=0 \text{ in } Q_T
\end{equation}
with the no-flux boundary conditions
\begin{equation}\label{A-2}
(1-x^{2})u_{x}=(1-x^{2})\left((1-x^{2})u_{x}\right)_{xx}=0 \text{ at } x=\pm1,\, t > 0,
\end{equation}
and the initial condition
\begin{equation}\label{A-3}
u(x,0)=u_{0}(x).
\end{equation}
Here $n > 0$, $Q_T = \Omega \times (0,T)$, $\Omega :=(-1,1)$,  and $T >0$.
Integrating the equation (\ref{A-1}) by using boundary conditions (\ref{A-2}), we obtain the mass conservation property
\begin{equation}\label{mass-con}
\int \limits_{\Omega}{ u(x,t) dx} = \int \limits_{\Omega}{ u_0(x ) dx} =: M > 0.
\end{equation}
Consider  initial data $u_0(x) \geqslant 0$ for all $x \in \bar{\Omega}$ satisfying
\begin{equation}\label{A-4}
\int \limits_{\Omega}{ \{  u^2_0(x) + (1-x^2) u^2_{0,x} (x)\} dx} < \infty  .
\end{equation}

\begin{definition}\label{def}[weak solution]
Let $n > 0$. A  function $u$ is a weak solution of
the problem (\ref{A-1})--(\ref{A-3}) with initial data $u_0$ satisfying (\ref{A-4})
if   $u(x,t)$ 
has the following properties
$$
(1-x^2)^{\frac{\beta}{2}} u \in C_{x,t}^{\frac{\alpha}{2}, \frac{\alpha}{8}}(\bar{Q}_T),
\ 0 < \alpha < \beta \leqslant \tfrac{2}{n} ,
$$
$$
u_t \in L^2(0,T; (H^1(\Omega))^*), \ (1-x^2)^{\frac{1}{2}}  u_x  \in L^{\infty}(0,T; L^2(\Omega)),
$$
$$
(1-x^2)^{\frac{1}{2}} |u|^{\frac{n}{2}} ( (1-x^2)u_x )_{xx}   \in L^2(P),
$$
$u$ satisfies (\ref{A-1}) in week sense:
$$
\int \limits_{0}^T {\langle u_t , \phi \rangle \,dt} -
\iint \limits_{P} { (1-x^2) |u|^{n} ( (1-x^2)u_x )_{xx} \phi_{x} \,dx dt}\\
 =0
$$
for all $\phi \in L^2(0,T; H^1(\Omega))$, where $P : = \bar{Q}_T \setminus
\{ \{u = 0\} \cup \{t =0\}\} $,
$$
(1-x^2)^{\frac{1}{2}} u_x(., t) \to (1-x^2)^{\frac{1}{2}} u_{0,x}(.)
\text{ strongly in } L^2(\Omega) \text{ as } t \to 0,
$$
and boundary conditions (\ref{A-2}) hold at all points of the lateral boundary, where
$\{u \neq 0\}$.
\end{definition}

Let us denote by
\begin{equation}\label{G_0-def}
0 \leqslant G_{0}(z): =   \begin{cases} \tfrac{ z^{2 -n} - A^{2 -n} }{(n-1)(n
- 2)} - \tfrac{A^{1-n}}{1 -n}(z - A) \text{
if } n \neq   1, 2  , \\
z \ln z - z(\ln A + 1) +A \text{ if }
n = 1 , \\
\ln (\tfrac{A}{z}) + \tfrac{z}{A} - 1 \text{ if } n = 2,
\end{cases}
\end{equation}
where $A = 0$ if $n \in (1,2)$ and $A > 0$ if else.

\begin{theorem}[strong solution]\label{Th1}
Assume that $n \geqslant 1$ and initial data $u_0$ satisfies $\int \limits_{\Omega}{G_0(u_0) \,dx} < +\infty $ then the problem (\ref{A-1})--(\ref{A-3}) has a non-negative weak solution, $u$, in the sense of Definition~\ref{def},
such that
$$
(1-x^2)u_x  \in L^2(0,T;H^1(\Omega)), \ (1-x^2)^{\frac{\gamma}{2}}u_x  \in L^2(Q_T), \ \gamma \in (0,1],
$$
$$
u   \in L^{\infty}(0,T; L^2(\Omega)),\ (1-x^2)^{\frac{\mu}{2}}u   \in L^2(Q_T),  \  \mu \in (-1, \beta].
$$
\end{theorem}

The existence of these solutions was proved in \cite{Tar-17}.
Our aim is to establish the finite speed of propagation property for a strong solution
$u$ of (\ref{A-1})  in  the sense of Theorem~\ref{Th1}.

\begin{theorem}[finite speed of propagation]\label{Th2}
Assume that $1 < n < 2$, the initial data satisfies the hypotheses of Theorem
\ref{Th1}  and the support of the initial data satisfies
$\mbox{supp}(h_0) \subset \Omega \setminus (-r_0,r_0)$, where $\Omega = (-1,1)$ and $r_0 \in (0,1)$. Let
$h$ be the strong solution from Theorem \ref{Th1}.
Then there exists a time $T^* > 0$ and a nondecreasing function $\Gamma(t) \in C([0,T^*])$,
$\Gamma(0) = 0$ such that $u$ has finite speed propagation, i.\,e.
$$
\mbox{supp}(u(\cdot,t)) \subseteq [-r_0 + \Gamma(t), r_0 - \Gamma(t)] \subset
\Omega
$$
for all $t \in [0,T^*]$. Moreover,  $\Gamma_{opt}(t)= C_0 t^{\frac{1}{n + 4}}$ for all $t \in [0,T^*]$.
\end{theorem}

\section{Proof of Theorem~\ref{Th2}}

\subsection{Local entropy estimate}

\begin{lemma}\label{local_entr}
Assume that $1 < n < 2$ and $\nu > 1$. Let $\zeta \in C^{1,2}_{t,x}(\bar{Q}_T)$ such that its support
satisfies $\text{supp}(\zeta) \subseteq \Omega$ and $(\zeta^4)_{x} = 0$ on
$\partial\Omega$. Then
there exist
positive constants $C_1,\, C_2$ are independent of $\Omega$, such
that for all $T > 0 $
the strong solution $u$ of
Theorem~\ref{Th1} satisfies
\begin{multline} \label{entr-06}
 \int \limits_{\Omega}{(1-x^2)^{\nu} \zeta^4(x,T)  G_{0} (u) \, dx} - \iint \limits_{Q_T}{ (1-x^2)^{\nu} (\zeta^4)_t  G_{0} (u) \, dxdt} + \\
 C_1 \iint \limits_{Q_T}
{ (1-x^2)^{\nu+2}  u_{xx}^2 \zeta^4 \,  dx dt} \leqslant \int \limits_{\Omega}{(1-x^2)^{\nu} \zeta^4(x,0) \, G_{0} (u_0) \, dx}\\
C_2 \iint \limits_{Q_T} { (1-x^2)^{\nu } u_{x}^2 [  \zeta^4  +
  \zeta^2 \zeta_x^2 + \zeta^3 |\zeta_{xx}|] \,  dx dt} + \\
   C_2 \iint \limits_{Q_T} {(1-x^2)^{\nu - 2} u^2 [  \zeta^4  +
 \zeta_x^4 +  \zeta^2 \zeta_{xx}^2 ]  \,  dx dt}  .
 \end{multline}
\end{lemma}

\begin{proof}[Proof of Lemma~\ref{local_entr}]

Equation (\ref{A-1}) is doubly degenerate when $u =0$ and $x = \pm 1$. Therefore, for
any $\epsilon > 0$ and $\delta > 0$ we consider two-parametric regularised equations
\begin{equation}\label{eq:regularized eqn}
 u_{\epsilon \delta,t} + \left[(1-x^{2}+\delta) (|u_{\epsilon\delta}|^{n}+ \epsilon  )\left((1-x^{2}+\delta) u_{\epsilon \delta,x}\right)_{xx}\right]_{x}=0 \text{ in } Q_T
\end{equation}
with boundary conditions
\begin{equation}\label{reg-1}
 u_{\epsilon \delta,x} =\left((1-x^{2}+\delta) u_{\epsilon \delta,x} \right)_{xx}=0  \text{ at } x= \pm1,
\end{equation}
and initial data
\begin{equation}\label{reg-2}
u_{\epsilon \delta}(x,0)= u_{0,\epsilon \delta }(x)\in C^{4 + \gamma}( \bar{\Omega}), \ \gamma > 0,
\end{equation}
where
\begin{equation}\label{reg-3-0}
u_{0,\epsilon\delta }(x) \geqslant u_{0\delta }(x) + \epsilon^{\theta}, \  \ \theta \in (0, \tfrac{1}{2(n-1)}),
\end{equation}
\begin{equation}\label{reg-3-1}
  u_{0 ,\epsilon \delta} \to   u_{0 \delta} \text{ strongly in } H^1(\Omega)
\text{ as } \epsilon \to 0,
\end{equation}
\begin{equation}\label{reg-3}
(1-x^{2}+\delta)^{\frac{1}{2}} u_{0x,\delta} \to (1-x^{2})^{\frac{1}{2}} u_{0,x } \text{ strongly in } L^2(\Omega)
\text{ as }  \delta \to 0.
\end{equation}
The parameters $\epsilon > 0$ and $\delta > 0$ in (\ref{eq:regularized eqn}) make  the problem
regular up to the boundary (i.e. uniformly parabolic).
The existence of a local in time solution of (\ref{eq:regularized eqn})
is guaranteed by the classical Schauder estimates (see \cite{friedman1958interior}).
Now suppose that $u_{\epsilon \delta}$ is a solution of equation
(\ref{eq:regularized eqn}) and that it is continuously differentiable
with respect to the time variable and fourth order continuously differentiable
with respect to the spatial variable. For the full detailed proof of existence of strong solutions
please refer to \cite{Tar-17}.

Multiplying the equation (\ref{eq:regularized eqn})
by $\phi(x,t) G'_{\epsilon}(u_{\epsilon\delta})$, integrating over $\Omega$, and then integrating by parts yield
\begin{multline} \label{entr-01}
 \tfrac{d}{dt} \int \limits_{\Omega}{\phi \, G_{\epsilon} (u_{\epsilon \delta}  ) \, dx} - \\
\int \limits_{\Omega}{\phi_t  G_{\epsilon} (u_{\epsilon \delta} ) \, dx} =
\int \limits_{\Omega} { (1-x^{2}+\delta) u_{\epsilon \delta,x} [(1-x^{2}+\delta) u_{\epsilon \delta,x} ]_{ xx} \phi \,  dx} +\\
\int \limits_{\Omega} { (1-x^{2}+\delta) (|u_{\epsilon\delta}|^n + \epsilon) G'_{\epsilon}(u_{\epsilon\delta}) [(1-x^{2}+\delta) u_{\epsilon \delta,x} ]_{ xx} \phi_x \,  dx} = \\
- \int \limits_{\Omega} {   [(1-x^{2}+\delta) u_{\epsilon \delta,x} ]_{  x}^2 \phi \,  dx} -
\int \limits_{\Omega} { (1-x^{2}+\delta) u_{\epsilon \delta,x} [(1-x^{2}+\delta) u_{\epsilon \delta,x} ]_{  x} \phi_x \,  dx} -\\
\int \limits_{\Omega} { [(1-x^{2}+\delta) (|u_{\epsilon\delta}|^n + \epsilon) G'_{\epsilon}(u_{\epsilon\delta}) \phi_x ]_x [(1-x^{2}+\delta) u_{\epsilon \delta,x} ]_{  x}  \,  dx}= \\
- \int \limits_{\Omega} {   [(1-x^{2}+\delta) u_{\epsilon \delta,x} ]_{  x}^2 \phi \,  dx}  + \tfrac{1}{2}\int \limits_{\Omega} { [(1-x^{2}+\delta) u_{\epsilon \delta,x} ]^2 \phi_{xx} \,  dx} - \\
\int \limits_{\Omega} {[(1-x^{2}+\delta) u_{\epsilon \delta,x} ]_{  x} (|u_{\epsilon\delta}|^n + \epsilon) G'_{\epsilon}(u_{\epsilon\delta})((1-x^{2}+\delta) \phi_x)_x \,  dx} -\\
 \int \limits_{\Omega} { [(1-x^{2}+\delta) u_{\epsilon \delta,x} ]_x (1-x^2+\delta)[  (|u_{\epsilon\delta}|^n + \epsilon) G'_{\epsilon}(u_{\epsilon\delta})]'_u u_{\epsilon \delta,x} \phi_x   \,  dx}.
\end{multline}
Integrating (\ref{entr-01}) in time and taking the regularizing parameter $\epsilon \to 0$, by applying the Young inequality
and $z^n G'_0(z) = \frac{1}{1-n} z  $, we finally get
\begin{multline} \label{entr-02}
 \int \limits_{\Omega}{\phi \, G_{0} (u_{ \delta}  ) \, dx} -
\iint \limits_{Q_T}{\phi_t  G_{0} (u_{ \delta} ) \, dxdt} + \iint \limits_{Q_T}
{   [(1-x^{2}+\delta) u_{\delta,x} ]_{  x}^2 \phi \,  dx dt} \leqslant \\
 \int \limits_{\Omega}{\phi \, G_{0} (u_{0, \delta}  ) \, dx} +
 \tfrac{1}{2}\iint \limits_{Q_T} { [(1-x^{2}+\delta) u_{ \delta,x} ]^2 \phi_{xx} \,  dx dt}  -\\
\tfrac{1}{1-n} \iint \limits_{Q_T}
{  [(1-x^{2}+\delta) u_{\delta,x} ]_{  x}  u_{\delta}   ((1-x^{2}+\delta) \phi_x)_x \,  dx dt} -\\
\tfrac{1}{1-n} \iint \limits_{Q_T}
{  [(1-x^{2}+\delta) u_{\delta,x} ]_{  x}(1-x^{2}+\delta) u_{ \delta,x} \phi_x  \,  dx dt}
\leqslant\\
 \int \limits_{\Omega}{\phi \, G_{0} (u_{0, \delta}  ) \, dx} +  \mu  \iint \limits_{Q_T}
{   [(1-x^{2}+\delta) u_{\delta,x} ]_{  x}^2 \phi \,  dx dt} +  \\
\tfrac{2-n}{2(1-n)}\iint \limits_{Q_T} { [(1-x^{2}+\delta) u_{ \delta,x} ]^2 \phi_{xx} \,  dx dt} +
\tfrac{1}{4\mu (1-n)^2}\iint \limits_{Q_T} { u^2_{\delta}  \tfrac{((1-x^{2}+\delta) \phi_x)_x^2}{\phi} \,  dx dt} ,
\end{multline}
where $\mu  > 0$. Choosing $\mu $ in (\ref{entr-02}) such that $0 < \mu  < 1$,
we arrive at
\begin{multline} \label{entr-03}
 \int \limits_{\Omega}{\phi \, G_{0} (u_{ \delta}  ) \, dx} -\\
\iint \limits_{Q_T}{\phi_t  G_{0} (u_{ \delta} ) \, dxdt} + C \iint \limits_{Q_T}
{   [(1-x^{2}+\delta) u_{\delta,x} ]_{  x}^2 \phi \,  dx dt} \leqslant \\
C \iint \limits_{Q_T} { [(1-x^{2}+\delta) u_{ \delta,x} ]^2   |\phi_{xx}|  \,  dx dt} +
C \iint \limits_{Q_T} { u^2_{\delta}  \tfrac{((1-x^{2}+\delta) \phi_x)_x^2}{\phi}   \,  dx dt} .
\end{multline}
Letting $\delta \to 0$ in (\ref{entr-03}), we deduce that
\begin{multline} \label{entr-04}
 \int \limits_{\Omega}{\phi(T) \, G_{0} (u) \, dx} -\iint \limits_{Q_T}{\phi_t  G_{0} (u) \, dxdt} + \\
C \iint \limits_{Q_T}
{   [(1-x^{2}) u_{x} ]_{  x}^2 \phi \,  dx dt} \leqslant \int \limits_{\Omega}{\phi(0) \, G_{0} (u_0) \, dx}\\
C \iint \limits_{Q_T} { [(1-x^{2}) u_{x} ]^2   |\phi_{xx}|   \,  dx dt} +
C \iint \limits_{Q_T} { u^2   \tfrac{((1-x^{2}) \phi_x)_x^2}{\phi}  \,  dx dt} .
\end{multline}
Taking $\phi(x,t) = (1-x^2)^{\nu} \zeta^4(x,t)$ in (\ref{entr-04})  for $\nu > 1 $,
we have
\begin{multline} \label{entr-05}
 \int \limits_{\Omega}{(1-x^2)^{\nu} \zeta^4(T)  G_{0} (u) \, dx} - \iint \limits_{Q_T}{ (1-x^2)^{\nu} (\zeta^4)_t  G_{0} (u) \, dxdt} + \\
 C \iint \limits_{Q_T}
{ (1-x^2)^{\nu}  [(1-x^{2}) u_{x} ]_{  x}^2 \zeta^4 \,  dx dt} \leqslant \int \limits_{\Omega}{(1-x^2)^{\nu} \zeta^4(0) \, G_{0} (u_0) \, dx}+\\
C \iint \limits_{Q_T} { [(1-x^{2}) u_{x} ]^2 [ (1-x^2)^{\nu - 2  } \zeta^4  +
(1-x^2)^{\nu} ( \zeta^2 \zeta_x^2 + \zeta^3 |\zeta_{xx}|) ] \,  dx dt} + \\
C \iint \limits_{Q_T} { u^2 [ (1-x^2)^{\nu - 2 } \zeta^4  +
(1-x^2)^{\nu + 2 }  \zeta_x^4 +
(1-x^2)^{\nu + 2} \zeta^2 \zeta_{xx}^2 ]  \,  dx dt} \leqslant \\
\int \limits_{\Omega}{(1-x^2)^{\nu} \zeta^4(0) \, G_{0} (u_0) \, dx}+
C \iint \limits_{Q_T} { (1-x^2)^{\nu} u_{x}^2 [  \zeta^4  +
  \zeta^2 \zeta_x^2 + \zeta^3 |\zeta_{xx}| ] \,  dx dt} + \\
   C \iint \limits_{Q_T} {(1-x^2)^{\nu - 2 } u^2 [  \zeta^4  +
 \zeta_x^4 +  \zeta^2 \zeta_{xx}^2 ]  \,  dx dt}  ,
 \end{multline}
whence we deduce   (\ref{entr-06}).
\end{proof}

\subsection{Finite speed of propagation}

For an arbitrary $s > 0$ and $0 < \delta \leqslant s$ we consider the families of sets
\begin{equation}\label{G:om}
\begin{gathered}
\Omega (s) := \{ x \in \bar{\Omega}: |x| \leqslant s \},\ Q_T(s)= (0,T)
\times \Omega (s), \\
 K_T(s,\delta) = Q_T(s) \setminus Q_T(s-\delta).
\end{gathered}
\end{equation}
We introduce a nonnegative cutoff function $\eta(\tau)$ from the
space $C^2(\mathbb{R}^1)$ with the following properties:
\begin{equation}\label{G:test1}
\eta(\tau) = \left\{
\begin{aligned}
\hfill 1 \  & \ \text{ if } \tau \leqslant 0,\\
\hfill -\tau^3 (6 \tau^2 - 15 \tau + 10) +1 \  & \ \text{ if } 0  < \tau < 1,\\
\hfill  0 \ & \ \text{ if } \tau \geqslant 1.
\end{aligned}\right.
\end{equation}
Next we introduce our main cut-off functions $\eta _{s,\delta }
(x) \in C^2 (\bar{\Omega})$ such that $0 \leqslant \eta_{s,\delta
} (x) \leqslant 1 \  \forall\,  x \in \bar{\Omega}$ and possess
the following properties:
\begin{equation}\label{G:test2}
\begin{gathered}
\eta _{s,\delta } (x) = \eta \left( \tfrac{|x| - (s -\delta)}{\delta} \right) = \left\{
\begin{aligned} \hfill  1 \;
& , x \in \Omega(s - \delta),\\
\hfill  0\; & , x \in \Omega \setminus \Omega(s ), \\
\end{aligned}\right. \\
| (\eta _{s,\delta })_x| \leqslant \tfrac{15}{8\delta } ,\
| (\eta _{s,\delta })_{xx}| \leqslant \tfrac{5(\sqrt{3} -1)} {\delta^2}
\end{gathered}
\end{equation}
for all $s > 0$ and  $0 < \delta \leqslant s  $. Choosing
$\zeta^4(x,t) = \eta _{s,\delta } (x) e^{-\tfrac{t}{T}}$ in
(\ref{entr-06}), we arrive at
\begin{multline}\label{entr-07}
\int\limits_{\Omega (s - \delta)} {(1-x^2)^{\nu} u^{2 - n}(T)\,dx} +
\tfrac{C}{T} \iint\limits_{Q_T(s - \delta)} {(1-x^2)^{\nu} u^{2   - n}\,dx dt} + \\
C \iint\limits_{Q_T(s - \delta)} {(1-x^2)^{\nu+2} u_{xx}^2 \,dx dt} \leqslant
e \int\limits_{\Omega (s)} {(1-x^2)^{\nu} u_0^{2 - n}(x)\,dx}  +\\
\tfrac{C}{\delta^2} \iint\limits_{K_T(s,\delta)} { (1-x^2)^{\nu } u_{x}^2 \,dx dt} +
\tfrac{C}{\delta^4} \iint\limits_{K_T(s,\delta)} { (1-x^2)^{\nu - 2 }  u^2 \,dx dt}
\end{multline}
for all $0 < \delta \leqslant s  $. By (\ref{entr-07}) we deduce that
\begin{multline*}
(1-(s-\delta)^2)^{ \nu} \int\limits_{\Omega (s - \delta)} { u^{2 - n}(T)\,dx} +
\tfrac{C(1-(s-\delta)^2)^{ \nu}}{T} \iint\limits_{Q_T(s - \delta)} {  u^{2   - n}\,dx dt} + \\
C(1-(s-\delta)^2)^{ \nu  }\iint\limits_{Q_T(s - \delta)} {(1-x^2)^{2} u_{xx}^2 \,dx dt} \leqslant
\tfrac{C(1-(s-\delta)^2)^{ \nu} }{\delta^2} \iint\limits_{K_T(s,\delta)} {  u_{x}^2 \,dx dt} + \\
\tfrac{C(1-(s-\delta)^2)^{ \nu} }{\delta^4} \iint\limits_{K_T(s,\delta)} { (1-x^2)^{ - 2 }  u^2 \,dx dt},
\end{multline*}
whence
\begin{multline}\label{entr-08}
 \int\limits_{\Omega (s - \delta)} { u^{2 - n}(T)\,dx} +
\tfrac{C }{T} \iint\limits_{Q_T(s - \delta)} {  u^{2   - n}\,dx dt} + \\
C (1-r_0^2)^2 \iint\limits_{Q_T(s - \delta)} { u_{xx}^2 \,dx dt} \leqslant
\tfrac{C }{\delta^2} \iint\limits_{K_T(s,\delta)} {  u_{x}^2 \,dx dt} + \\
\tfrac{C (1- r_0^2)^{ - 2 } }{\delta^4} \iint\limits_{K_T(s,\delta)} {  u^2 \,dx dt} =: R(s)
\end{multline}
for all $0 < \delta \leqslant s \leqslant r_0 $. We apply Lemma~\ref{A.4}
in the region $\Omega(s-\delta)$ to a function $v := u$ with $ a = d = j =2$,
$b = 2 - n $, $k = 0$ (or $k =1$), $N = 1$, and
$\theta_1 =  \frac{n}{8 -3n}$ (or $\theta_2 =  \frac{4- n}{8 -3n}$).
Integrating the resulted inequalities with respect to time and
taking into account (\ref{entr-08}), we arrive at the following
relations:
\begin{equation}\label{entr-09}
A(s - \delta)  \leqslant C (1-r_0^2)^{-\alpha_1} T^{\beta_1} \bigl(R(s) \bigr)^{1 +  \kappa_1}   + C\, T \bigl(R(s)  \bigr)^{1 +  \kappa_3} ,
\end{equation}
\begin{equation}\label{entr-10}
B(s - \delta)  \leqslant C (1-r_0^2)^{-\alpha_2} T^{\beta_2} \bigl(R(s) \bigr)^{1 +  \kappa_2}   + C\, T \bigl(R(s)  \bigr)^{1 +  \kappa_3},
\end{equation}
where
$$
A(s) := \iint\limits_{Q_T(s )}{ u^{2} dxdt}, \  B(s) := \iint\limits_{Q_T(s - \delta)}{ u_x^{2} dxdt},
$$
$$
\alpha_1 = \tfrac{4(n+4)}{8-3n},\ \alpha_2 =  \tfrac{4(6 -n)}{8-3n}, \beta_1 =  \tfrac{4(2 -n)}{8-3n}, \
\beta_2 =  \tfrac{2(2 -n)}{8-3n},
$$
$$
\kappa_1 = \tfrac{4n}{8-3n}, \ \kappa_2 = \tfrac{2n}{8-3n}, \ \kappa_3 = \tfrac{n}{2- n}.
$$
Since all integrals on the right-hand sides of (\ref{entr-09}), (\ref{entr-10})
vanish as $T \to 0$ and $u \in L^2(0,T; H^1(-r_0,r_0))$, then
for sufficiently small $T$ we get
\begin{equation}\label{entr-09-2}
A(s - \delta)  \leqslant C_3  (1-r_0^2)^{- \alpha_1} T^{\beta_1} \bigl( \delta^{-4} A(s) + \delta^{-2} B(s) \bigr)^{1 + \kappa_1}   ,
\end{equation}
\begin{equation}\label{entr-10-2}
B(s - \delta)  \leqslant C_4  (1-r_0^2)^{- \alpha_2} T^{\beta_2} \bigl( \delta^{-4} A(s) + \delta^{-2} B(s) \bigr)^{1 + \kappa_2}.
\end{equation}
Let us denote by
$$
D(s) := A^{1 + \kappa_2}(s ) + B^{1 + \kappa_1}(s), \ \kappa = (1+\kappa_1)(1+\kappa_2),
$$
$$
C_5(T) := 2^{ \kappa -1}\max \{ [ C_3 (1-r_0^2)^{- \alpha_1} T^{\beta_1}]^{1+\kappa_2}, [ C_4 (1-r_0^2)^{- \alpha_2} T^{\beta_2}  ]^{1+ \kappa_1} \}.
$$
Without loss of generality, we can define the function
$$
\tilde{D}(s) = D(s) \text{ if } s \in (0,r_0], \text{ and } \tilde{D}(s) = 0 \text{ if } s > r_0.
$$
Then by (\ref{entr-09-2}), (\ref{entr-10-2}) we arrive at
\begin{equation}\label{entr-11}
\tilde{D}(s - \delta)  \leqslant C_5(T)  \bigl( \delta^{-4 \kappa } \tilde{D}^{1+\kappa_1}(s) +
 \delta^{-2\kappa} \tilde{D}^{1+\kappa_2}(s) \bigr)
\end{equation}
for all $s \in \mathbb{R}^{+}$ and $\delta \in (0,r_0]$. Choosing
$$
\delta (s) = \max \{ [4 C_5(T) \tilde{D}^{ \kappa_1}(s) ]^{\frac{1}{4\kappa}}, [4 C_5(T) \tilde{D}^{ \kappa_2}(s) ]^{\frac{1}{2\kappa}} \}
$$
in (\ref{entr-11}), we find that
$$
\tilde{D}(s - \delta(s) )  \leqslant \tfrac{1}{2} \tilde{D}(s),
$$
whence it follows
\begin{equation}\label{entr-12}
\delta (s - \delta(s) )  \leqslant \gamma \delta(s)
 \  \ \forall\, s \in \mathbb{R}^{+},
\end{equation}
where $\gamma = \max\{2^{-\frac{\kappa_1}{4\kappa}}, 2^{-\frac{\kappa_2}{2\kappa}} \} < 1$.
Applying Stampacchia's lemma (see Lemma~\ref{A.5}) to (\ref{entr-12}), we obtain that
$$
\delta (s) = 0 \text{ for all } s \leqslant r_0 -  \tfrac{\delta(r_0)}{1-\gamma}.
$$
Next, we will find the upper bound for $\delta(r_0)$. In view of Theorem~\ref{Th1},
$(1-x^2)^{\frac{\nu -2}{2}} u \in L^2(Q_T)$ and $(1-x^2)^{\frac{\nu}{2}} u_x \in L^2(Q_T)$
for any $\nu > 1$ then the right-hand side of (\ref{entr-07}) is bounded for all $T > 0$.
So, taking $s = 2r_0$ and $\delta = r_0$ in (\ref{entr-09}) and (\ref{entr-10}), we obtain that
$ \tilde{D}(r_0) \leqslant C_6 \, C_5(T)$, whence
$$
\delta (r_0) \leqslant C_7 (1-r_0^2)^{- \frac{2(6-n)}{8-3n}} T^{\frac{ 2 -n }{8-3n}}.
$$
This implies the upper bound for speed of propagation to solution  support, i.\,e.
\begin{equation}\label{fsp-0}
\Gamma(T) \leqslant r_0 - C_8 T^{\frac{ 2 -n }{8-3n}}  \ \forall\, T \leqslant T^*:= (\tfrac{r_0}{C_8})^{\frac{8-3n}{ 2 -n }}
\end{equation}
for
any $r_0 \in (0,1)$, where $C_8 = \frac{C_7}{1-\gamma} (1-r_0^2)^{- \frac{2(6-n)}{8-3n}}$.

\subsection{Exact upper bound for speed of propagation}

In this section we refine the estimate (\ref{fsp-0}). Applying Lemma~\ref{A.4}
in the region $\Omega(s) \setminus \Omega(s-\delta)$ to a function $v := u$ with $ a = d = j =2$,
$b = 1 $, $k = 0$ (or $k =1$), $N = 1$, and
$\theta_1 =  \frac{1}{5}$ (or $\theta_2 =  \frac{3}{5}$), and integrating the resulted
inequalities with respect to time,
taking into account the mass conservation (\ref{mass-con}), we arrive at the following
estimates:
\begin{equation}\label{u-1}
\iint\limits_{K_T(s,\delta)} {  u^2 \,dx dt} \leqslant C\,T^{1-\theta_1} M^{2(1-\theta_1)}
\Bigl( \iint\limits_{K_T(s,\delta)} {  u_{xx}^2 \,dx dt} \Bigr)^{\theta_1} + C\,\delta^{-1} T M^2,
\end{equation}
\begin{equation}\label{u-2}
\iint\limits_{K_T(s,\delta)} {  u_x^2 \,dx dt} \leqslant C\,T^{1-\theta_2} M^{2(1-\theta_2)}
\Bigl( \iint\limits_{K_T(s,\delta)} {  u_{xx}^2 \,dx dt} \Bigr)^{\theta_2} + C\,\delta^{-3} T M^2.
\end{equation}
Using (\ref{u-1}), (\ref{u-2}) and Young inequality, from (\ref{entr-08}) we find that
\begin{multline*}
 \int\limits_{\Omega (s - \delta)} { u^{2 - n}(T)\,dx} +
\tfrac{C }{T} \iint\limits_{Q_T(s - \delta)} {  u^{2   - n}\,dx dt} +
C (1-r_0^2)^2 \iint\limits_{Q_T(s - \delta)} { u_{xx}^2 \,dx dt} \leqslant \\
\varepsilon (1-r_0^2)^2 \iint\limits_{K_T(s,\delta)} { u_{xx}^2 \,dx dt} +
C_{\varepsilon} \delta^{-5} (1-r_0^2)^{-3} T M^2,
\end{multline*}
where $\varepsilon > 0$. Selecting $\varepsilon \in (0,2^{-5})$ enough small and making
standard iteration process, we get
\begin{multline}\label{u-3}
 \int\limits_{\Omega (s - \delta)} { u^{2 - n}(T)\,dx} +
\tfrac{C }{T} \iint\limits_{Q_T(s - \delta)} {  u^{2   - n}\,dx dt} + \\
C (1-r_0^2)^2 \iint\limits_{Q_T(s - \delta)} { u_{xx}^2 \,dx dt} \leqslant
C\,\delta^{-5} (1-r_0^2)^{-3} T M^2.
\end{multline}
Taking $s = 2\Gamma(T)$ and $\delta = \Gamma(T)$ in (\ref{u-3}), we obtain that
$$
\iint\limits_{Q_T(\Gamma(T))} { u_{xx}^2 \,dx dt} \leqslant C\,\Gamma^{-5}(T) (1-r_0^2)^{-5} T M^2,
$$
whence, similar to (\ref{u-1}) and (\ref{u-2}), we have
$$
A(\Gamma(T)) \leqslant C\,\Gamma^{-1}(T) (1-r_0^2)^{-1} T M^2, \ \  B(\Gamma(T)) \leqslant C\,\Gamma^{-3}(T) (1-r_0^2)^{-3} T M^2.
$$
Hence,
\begin{multline*}
\delta(\Gamma(T)) \leqslant C  \max \biggl \{ [\Gamma^{- \kappa_1 }(T) (1- r_0^2)^{-(\kappa_1 + \alpha_1)}
T^{ \kappa_1 + \beta_1 } M^{2\kappa_1}]^{\frac{1}{4(1+\kappa_1)}},    \\
 [\Gamma^{- 3\kappa_2 }(T) (1- r_0^2)^{-(3\kappa_2 + \alpha_2)}
T^{ \kappa_2 + \beta_2 } M^{2\kappa_2}]^{\frac{1}{2(1+\kappa_2)}}  \biggr \} = \\
C_9 \max \biggl \{ \Gamma^{-  \frac{n}{n + 8} }(T)
T^{ \frac{2}{n + 8}}  ,      \Gamma^{-  \frac{3n}{ 8 -n} }(T)
T^{ \frac{2}{8 - n}}   \biggr  \}.
  \end{multline*}
Thus, we have
\begin{equation}\label{u-4}
\Gamma(T) + C_{10} \max \biggl \{ \Gamma^{-  \frac{n}{n + 8} }(T)
T^{ \frac{2}{n + 8}}  ,      \Gamma^{-  \frac{3n}{ 8 -n} }(T)
T^{ \frac{2}{8 - n}}   \biggr  \} \leqslant r_0,
\end{equation}
where $C_{10} = \frac{C_9}{1-\gamma}$. Now we use the following calculus result: let $a > 0$ and $b > 0$
then the function $f(x) = x + a\,x^{-b}$ for all $x \geqslant 0$  has minimum at $x_{\min} = (a b)^{\frac{1}{1+b}}$
and $f(x_{\min}) = \frac{1+b}{b} x_{\min}$. Hence, minimizing the right-hand side,
we obtain that
$$
\Gamma_{opt}(T) = C_{0}  T^{\frac{1}{n + 4}} \ \ \forall\, T \leqslant T^*.
$$
This proves Theorem~\ref{Th1} completely. $\square$

\section*{Appendix A}\label{app-A}

\renewcommand{\thesection}{A}\setcounter{lemma}{0}
\renewcommand{\thetheorem}{A.\arabic{theorem}}

\begin{lemma}[\cite{N1}]\label{A.4}
If $\Omega  \subset \mathbb{R}^N $ is a bounded domain
with piecewise-smooth boundary, $a > 1$, $b \in (0, a),\ d
> 1,$ and $0 \leqslant k < j,\ k,j \in \mathbb{N}$, then there
exist positive constants $d_1$ and $d_2$ $(d_2 = 0 \text{ if }
\Omega$ is unbounded$)$ depending only on $\Omega ,\ d,\ j,\ b,$ and
$N$ such that the following inequality is valid for every $v(x) \in
W^{j,d} (\Omega ) \cap L^b (\Omega )$:
$$
\left\| {D^k v} \right\|_{L^a (\Omega )}  \leqslant d_1 \left\| {D^j
v} \right\|_{L^d (\Omega )}^\theta  \left\| v \right\|_{L^b (\Omega
)}^{1 - \theta }  + d_2 \left\| v \right\|_{L^b (\Omega )} ,\ \theta
= \frac{{\tfrac{1} {b} + \tfrac{k} {N} - \tfrac{1} {a}}} {{\tfrac{1}
{b} + \tfrac{j} {N} - \tfrac{1} {d}}} \in \left[ {\tfrac{k} {j},1}
\right)\!\!.
$$
 Note that if  $\Omega = B(0, R) \setminus B(0,r) $, where $B(0,x)$ is
ball with the radius $x$ and the origin at $0$, then $d_2 = c (R - r)^{ -   \frac{(a - b)N}{a b} -k}$.
\end{lemma}

%
%
%

\begin{lemma}[\cite{Shi93}]\label{A.5}
Assume that $f(s)$ is nonnegative
nondecreasing function  satisfying the following
inequality
$$
f(s - f(s)) \leqslant \varepsilon  f(s) \ \forall\, s \leqslant s_0,
$$
where $\varepsilon \in (0,1)$. Then $f(s) = 0$ for all $s \leqslant s_0 - \frac{f(s_0)}{1-\varepsilon}$.
\end{lemma}


\begin{thebibliography}{99}

\bibitem{beretta1995nonnegative}
\newblock E. Beretta, M. Bertsch, and R. Dal Passo.
\newblock Nonnegative solutions of a fourth-order nonlinear degenerate parabolic equation,
\newblock \emph{ Archive for rational mechanics and analysis},  {129}(2): 175--200, 1995.

\bibitem{bernis1990higher}
\newblock F. Bernis, A. Friedman.
\newblock Higher order nonlinear degenerate parabolic equations,
\newblock \emph{  J. Differential Equations},  {83}(1): 179--206, 1990.

\bibitem{B6}
F. Bernis.
\newblock Finite speed of propagation and continuity of the interface for thin viscous flows.
\newblock {\em Adv. Differential Equations}, 1(3): 337--368, 1996.

\bibitem{B7}
F. Bernis.
\newblock Finite speed of propagation for thin viscous flows when {$2\leq
  n<3$}.
\newblock {\em Comptes Rendus de l'Acad\'emie des Sciences. S\'erie I.
              Math\'ematique}, 322(12): 1169--1174, 1996.

\bibitem{BPW-92}
F. Bernis, L.A. Peletier and S. M. Williams.
\newblock Source type solutions of a fourth order nonlinear degenerate parabolic equation.
\newblock {\em Nonlinear Anal.}, 18: 217--234, 1992.


\bibitem{bertozzi1994singularities}
\newblock Andrea L. Bertozzi  et al.
\newblock  \emph{Singularities and similarities in interface flows,}
\newblock Trends and perspectives in applied mathematics. Springer New York, 155--208, 1994.



%


\bibitem{Ch10}
\newblock Marina Chugunova, Mary C. Pugh, and Roman M. Taranets.
\newblock Nonnegative solutions for a long-wave unstable
   thin film equation with convection,
\newblock  \emph{ SIAM Journal on Mathematical Analysis},  42(4): 1826--1853, 2010.

\bibitem{CT-12}
\newblock Marina Chugunova, and Roman M. Taranets.
\newblock Qualitative analysis of coating flows on a rotating horizontal cylinder,
\newblock  \emph{ International Journal of Differential Equations},  2012:
 Article ID 570283, 30 pages, 2012.

\bibitem{friedman1958interior}
\newblock Avner Friedman.
\newblock  Interior estimates for parabolic systems of partial differential equations,
\newblock  \emph{ J. Math. Mech.}, {7}(3): 393--417, 1958.

\bibitem{HulShish}
Josephus Hulshof, and Andrey E. Shishkov.
\newblock The thin film equation with {$2\leq n<3$}: finite speed of
              propagation in terms of the {$L^1$}-norm.
\newblock {\em Adv. Differential Equations}, 3(5):625--642, 1998.

\bibitem{kangcoatingsphere}
\newblock D. Kang,  A. Nadim, and M. Chugunova.
\newblock  Dynamics and equilibria of thin viscous coating films on a rotating sphere,
\newblock  \emph{ Journal of Fluid Mechanics},  791: 495--518, 2016. 

\bibitem{KNC-17}
\newblock D. Kang,  A. Nadim, and M. Chugunova.
\newblock  Marangoni effects on a thin liquid film coating a sphere with axial or radial thermal
gradients,
\newblock  \emph{ Physics of Fluids},  29: 072106-1--072106-15, 2017.

\bibitem{KSZ-16}
\newblock D. Kang,  Tharathep Sangsawang and Jialun Zhang.
\newblock  Weak solution of a doubly degenerate parabolic equation,
\newblock  \emph{  arXiv:1610.06303v2},  2017.



\bibitem{N1}
L.~Nirenberg.
\newblock An extended interpolation inequality.
\newblock {\em Ann. Scuola Norm. Sup. Pisa (3)}, 20: 733--737, 1966.

\bibitem{Shi93}
A.~E.~Shishkov.
\newblock Dynamics of the geometry of the support of the generalized
solution of a higher-order quasilinear parabolic equation in divergence form,
\newblock {\em Differ. Uravn.}, 29(3): 537--547, 1993.

\bibitem{Tar-17}
Roman M. Taranets.
\newblock  Strong solutions of the thin film equation in spherical geometry.
\newblock {\em arXiv:1709.10496}, 2017.

\bibitem{TH-10}
\newblock D. Takagi, and Herbert E. Huppert.
\newblock  Flow and instability of thin films on a cylinder and sphere,
\newblock  \emph{ Journal of Fluid Mechanics},  647: 221--238, 2010.


\bibitem{Wil-94}
\newblock S.K. Wilson.
\newblock  The onset of steady Marangoni convection in a spherical geometry,
\newblock  \emph{ Journal of Engineering Mathematics},  {28}: 427--445, 1994.



\end{thebibliography}
\end{document}